\documentclass{amsart}

\allowdisplaybreaks

\usepackage{microtype}
\usepackage{hyperref}
\usepackage{amsmath}
\usepackage{amssymb}
\usepackage{amsopn}
\usepackage{psfrag}
\usepackage{mathrsfs}
\usepackage[all]{xy} 

\usepackage[utf8]{inputenc}
\usepackage{tikz}

\newcommand{\triplerightarrow}{%
\tikz[minimum height=0ex]
  \path[->]
   node (a)            {}
   node (b) at (1em,0) {}
  (a.north)  edge (b.north)
  (a.center) edge (b.center)
  (a.south)  edge (b.south);%
}

\xyoption{2cell} 

\UseTwocells

\newcommand{\BB}{\mathcal{B}}
\newcommand{\CC}{\mathcal{C}}
\newcommand{\DD}{\mathcal{D}}
\newcommand{\EE}{\mathcal{E}}
\newcommand{\GG}{\mathcal{G}}
\newcommand{\HH}{\mathcal{H}}
\newcommand{\KK}{\mathcal{K}}
\newcommand{\TT}{\mathcal{T}}
\newcommand{\XX}{\mathcal{X}}
\newcommand{\YY}{\mathcal{Y}}
\newcommand{\ZZ}{\mathcal{Z}}

\newcommand{\sA}{\mathscr{A}}
\newcommand{\sB}{\mathscr{B}}

\newcommand{\define}{\textbf}

\newcommand{\onto}{\twoheadrightarrow}

\newcommand{\lact}{\triangleright}
\newcommand{\ract}{\triangleleft}

\DeclareMathOperator{\Ad}{Ad}

\DeclareMathOperator{\Hom}{Hom}
\DeclareMathOperator{\Aut}{Aut}
\DeclareMathOperator{\id}{id}

\DeclareMathOperator{\Pontr}{Pontr}

\DeclareMathOperator{\op}{op}

\DeclareMathOperator{\End}{End}
\DeclareMathOperator{\st}{st}
\DeclareMathOperator{\Cat}{Cat}
\DeclareMathOperator{\rev}{rev}
\DeclareMathOperator{\Mod}{Mod}

\DeclareMathOperator{\bicolim}{bicolim}

\newtheorem{thm}{Theorem}[section]
\newtheorem{prop}[thm]{Proposition}
\newtheorem{cor}[thm]{Corollary}
\newtheorem{lem}[thm]{Lemma}

\theoremstyle{definition}
\newtheorem{defn}[thm]{Definition}

\newtheorem{remark}[thm]{Remark}

\numberwithin{equation}{thm}

\title{Extensions of groups by braided 2-groups}
\author{Evan Jenkins}
\address{Department of Mathematics, University of Chicago, Chicago, IL 60637}
\email{ejenkins@math.uchicago.edu}

\begin{document}

\begin{abstract}
We classify extensions of a group $G$ by a braided 2-group $\mathcal{B}$ as defined by Drinfeld, Gelaki, Nikshych, and Ostrik. We describe such extensions as homotopy classes of maps from the classifying space of $G$ to the classifying space of the 3-group of braided $\mathcal{B}$-bitorsors. The Postnikov system of the latter space contains a generalization of the classical Pontryagin square to the setting of local coefficients, which has been previously discussed by Baues and more recently, in a setting close to ours, by Etingof, Nikshych, and Ostrik. We give an explicit cochain-level description of this Pontryagin square for group cohomology.
\end{abstract}

\maketitle  

\section{Introduction}

A classical problem in homological algebra is to classify extensions of a group $G$ by another group $K$, i.e., groups $E$ equipped with a surjection $\partial: E \onto G$ and an identification $K \cong \partial^{-1}(e)$. When $K = A$ is an abelian group, the action of an element of $E$ on an element of $A$ by conjugation depends only on its image in $G$, so extensions of $G$ by $A$ include the data of an action of $G$ as automorphisms of $A$. If we fix such an action, then extensions with that action are classified by $H^2(G, A)$ (see, for example, \cite[6.6]{weibel}).

In \cite[Appendix E]{dgno}, a categorification of the notion of an extension of a group by an abelian group is defined, in which the abelian group $A$ is replaced by a braided 2-group $\BB$ (see \cite{baezlauda} for basic results about 2-groups, which have also been studied under the names gr-category and categorical group, and \cite{joyalstreet} for basic results about braided monoidal categories). Given such an extension one has an underlying action of $G$ as braided autoequivalences of the braided 2-group $\BB$, and also an underlying extension of $G$ by the abelian group $A = \pi_0(\BB)$ of isomorphism classes of objects of $\BB$. It is natural to ask under what circumstances these two pieces of data will determine an extension of $G$ by $\BB$ and how unique such an extension is.

In this paper, we show that we can lift these data to an extension if and only if a certain cohomology class in $H^4(G, H)$ vanishes, where $H = \pi_1(\BB)$ is the abelian group of isomorphisms of the unit object in $\BB$, and the action of $G$ on $H$ is induced from the action of $G$ on $\BB$. This cohomology class comes from a function $H^2(G, A) \to H^4(G, H)$ (determined by a fixed braided 2-group $\BB$ and an action of $G$ on $\BB$) that generalizes the classical Pontryagin square, which can be recovered in our language by taking the trivial action of $G$ on $\BB$. If this obstruction vanishes, then extensions form a torsor for $H^3(G, H)$.

These results are parallel to the results of \cite{eno} on the related problem of classifying braided $G$-crossed fusion categories. In that paper, a braided $G$-crossed fusion category $\CC$ is viewed as a $G$-indexed family of invertible bimodules over the neutral component. Analogously, we will view extensions of $G$ by a braided 2-group $\BB$ as $G$-indexed families of braided bitorsors over $\BB$. We then use obstruction theory to obtain our classification.

Sections 2 through 4 set up the machinery of braided bitorsors. Section 5 relates extensions and bitorsors by a construction of Grothendieck. Section 6 contains our classification result. Section 7 contains an explicit cochain-level description of the Pontryagin square defined in Section 6.

The author would like to thank his advisor, Vladimir Drinfeld, for suggesting this topic and providing inspiration, guidance, and careful reading of many drafts. The author would also like to thank Peter May, Daniel Sch\"appi, Mike Shulman, and Ross Street for helpful discussions.

\section{Torsors for 2-groups}

We can define torsors for 2-groups much the same as we do for groups. We first review the notion of modules for monoidal categories.

A monoidal category $\CC$ can be viewed as a one-object bicategory, which we will denote by $\CC[1]$. This can be viewed as the ``delooping'' of $\CC$, a notion we will revisit in Section \ref{obstruction}. For now, we note only that strong monoidal functors between monoidal categories correspond to pseudofunctors between their deloopings, and that $\CC[1]^{\op} = (\CC^{\rev})[1]$, where $\CC^{\rev}$ denotes the category $\CC$ with reversed tensor product.

\begin{defn}
Let $\mathcal{C}$ be a monoidal category. A \define{left (resp. right) module} (or \define{module category}) over $\mathcal{C}$ is a pseudofunctor $\mathcal{X}: \CC[1] \to \Cat$ (resp. $\mathcal{X}: \CC[1]^{\op} \to \Cat$). By abuse of notation, we denote the image of the unique object of $\CC[1]$ by $\mathcal{X}$. We denote by ${_\mathcal{C}}\Mod$ and $\Mod_{\mathcal{C}}$ the 2-categories of left and right $\mathcal{C}$-modules, respectively.
\end{defn}

We will write a left action of an element $c \in \CC$ on an element $x \in \XX$ by $c \lact x$, and a right action as $x \ract c$. For the rest of this subsection, we will only deal with left $\mathcal{C}$-modules. Completely analogous definitions and proofs work with ``left'' everywhere replaced by ``right.''

\begin{lem}
\label{reflect}
Let $F: \XX \to \YY$ be a morphism of left $\CC$-modules, and suppose further that $F$ is an equivalence of categories. Then $F$ is an equivalence of $\CC$-modules.
\end{lem}

\begin{proof}
The forgetful 2-functor ${_\CC}\Mod \to \Cat$ is monadic, and hence it reflects adjoint equivalences (see \cite{lecreurer}).
\end{proof}

Recall that a 2-group is a monoidal category in which all objects and morphisms are invertible.

\begin{defn}
Let $\GG$ be a 2-group. A left $\mathcal{G}$-module $\mathcal{X}$ is a \define{left $\mathcal{G}$-torsor} if $\XX$ is nonempty and the characteristic map $\chi = (a, \pi_2): \mathcal{G} \times \mathcal{X} \to \mathcal{X} \times \mathcal{X}$ is an equivalence.
\end{defn}

\begin{defn}
The \define{trivial left $\mathcal{G}$-module}, denoted by $\mathcal{G}$, is $\mathcal{G}$ equipped with the action of left multiplication.
\end{defn}

\begin{defn}
The action of $\GG$ on $\XX$ is \define{essentially simply transitive} if, for every $x \in \XX$, the map $\GG \to \XX$, $g \mapsto g \lact x$ is an equivalence of categories.
\end{defn}

\begin{thm}
\label{2torsors}
Let $\mathcal{G}$ be a 2-group, $\mathcal{X}$ a left $\mathcal{G}$-module. The following are equivalent.
\begin{enumerate}
\item \label{2a} $\mathcal{X}$ is a left $\mathcal{G}$-torsor.
\item \label{2b} $\XX$ is nonempty, and the action of $\mathcal{G}$ on $\mathcal{X}$ is essentially simply transitive.
\item \label{2c} $\mathcal{X}$ is equivalent to the trivial left $\mathcal{G}$-module.
\end{enumerate}
\end{thm}

\begin{proof}
\eqref{2a} $\Rightarrow$ \eqref{2b}:\\
Fix $x \in \XX$. We will show that the map $F: g \mapsto g \lact x$ is full, faithful, and essentially surjective.

If $x' \in \mathcal{X}$, essential surjectivity of the characteristic map implies that there exists $g \in \mathcal{G}$, $x'' \cong x$ such that $g \lact x'' \cong x'$. It follows that $g \lact x \cong g \lact x'' \cong x'$. Thus, $F$ is essentially surjective.

If $g, g' \in \mathcal{G}$, $\phi_1, \phi_2: g \to g'$, such that $\phi_1 \lact x = \phi_2 \lact x: g \lact x \to g' \lact x$, then $(\phi_1 \lact x, x) = (\phi_2 \lact x, x): (g \lact x, x) \to (g' \lact x, x)$, so faithfulness of the characteristic map implies that $(\phi_1, x) = (\phi_2, x)$, and hence $\phi_1 = \phi_2$. Thus, $F$ is faithful.

Finally, if $g, g' \in \mathcal{G}$, and $f: g \lact x \to g' \lact x$, then fullness of the characteristic map gives a map $(\phi, \id): (g, x) \to (g', x)$ such that $\phi \lact \id: g \lact x \to g' \lact x$ is $f$. Thus, $F$ is full.

\eqref{2b} $\Rightarrow$ \eqref{2c}:\\
The map $g \mapsto g \lact x$ is a map $\GG \to \XX$ of $\GG$-modules and an equivalence of categories, so by Lemma \ref{reflect}, it is an equivalence of $\GG$-modules.

\eqref{2c} $\Rightarrow$ \eqref{2a}:\\
It suffices to show that the trivial left $\mathcal{G}$-module is a left $\mathcal{G}$-torsor. The map $(a, \pi_2): \mathcal{G} \times \mathcal{G} \to \mathcal{G} \times \mathcal{G}$ has quasi-inverse $(a \circ (-^{-1}), \pi_2)$, so $\mathcal{G}$ is a left $\mathcal{G}$-torsor.
\end{proof}

\begin{cor}
\label{mortor}
Any morphism $F: \mathcal{X} \to \mathcal{Y}$ of left $\mathcal{G}$-torsors is an equivalence.
\end{cor}

\begin{proof}
Fix $x \in \XX$. Then the following diagram commutes up to isomorphism.
\[\xymatrix{
& \GG \ar[dl]_{g \mapsto g \lact x} \ar[dr]^{g \mapsto g \lact F(x)}\\
\XX \ar[rr]_{F} & & \YY
}\]
 By Theorem \ref{2torsors}, the top two arrows are equivalences, so $F$ is an equivalence.
\end{proof}

\begin{prop}
\label{auttorsor}
The autoequivalence 2-group of a left $\mathcal{G}$-torsor is equivalent to $\mathcal{G}$.
\end{prop}

\begin{proof}
By Theorem \ref{2torsors}, every $\mathcal{G}$-torsor is equivalent to the trivial left $\mathcal{G}$-torsor, so it suffices to show this for the trivial left $\mathcal{G}$-torsor $\mathcal{G}$. We get a monoidal functor $\mathcal{G} \to \Aut_{\mathcal{G}}(\mathcal{G})$ by sending $g \in \mathcal{G}$ to $x \mapsto xg^{-1}$. We get a monoidal functor $\Aut_{\mathcal{G}}(\mathcal{G}) \to \mathcal{G}$ by sending $f \in \Aut_{\mathcal{G}}(\mathcal{G})$ to $f(e)^{-1}$. The composition $\mathcal{G} \to \Aut_{\mathcal{G}}(\mathcal{G}) \to \mathcal{G}$ is isomorphic to the identity; for the other direction, we note that if $f(e)^{-1} \cong g$, then $f(x) \cong xx^{-1}f(x) \cong xf(e) \cong xg^{-1}$ is an isomorphism of $\GG$-module morphisms.
\end{proof}

\section{Bitorsors for 2-groups}

We can also define bitorsors for 2-groups much the same as we do for groups. We begin here with the notion of bimodules for monoidal categories.

\begin{defn}
Let $\mathcal{C}$ and $\mathcal{D}$ be monoidal categories. A \define{$(\mathcal{C}, \mathcal{D})$-bimodule} is a $(\mathcal{D}^{\rev} \times \mathcal{C})$-module. We think of $\mathcal{C}$ as acting on the left and $\mathcal{D}$ as acting on the right.\end{defn}

We denote by ${_\mathcal{C}}\Mod_\mathcal{D}$ the 2-category of $(\mathcal{C}, \mathcal{D})$-bimodules.

\begin{defn}
Let $\CC$, $\DD$, and $\EE$ be 2-groups, and let $\XX$ and $\YY$ be $(\CC, \DD)$- and $(\DD, \EE)$-bimodules, respectively. The \define{tensor product} $\XX \times_{\DD} \YY$ of $\XX$ and $\YY$ is defined to be the bicolimit of the diagram
\[\XX \times \DD \times \DD \times \YY \triplerightarrow \XX \times \DD \times \YY \rightrightarrows \XX \times \YY,\]
where $\CC$ acts on the left on each left factor, and $\EE$ acts on the right on each right factor. Here, the arrows correspond to the product of $\DD$ and the various actions of $\DD$ on $\XX$ and $\YY$.
\end{defn}

In other words, a $(\CC, \EE)$-bimodule map $\XX \times_{\DD} \YY \to \ZZ$ is a map $F: \XX \times \YY \to \ZZ$ equipped with a family of isomorphisms $\phi_d: F(x \ract d, y) \stackrel{\cong}{\to} F(x, d \lact y)$ natural in $d$ and compatible with the product in $\DD$. 

The following important ``folklore'' result is entirely formal, although the author does not believe a reference exists.

\begin{thm}
Let $\CC$ be a monoidal category. Then $\CC$-bimodules form a monoidal 2-category.
\end{thm}

\begin{proof}
We give only the vaguest outline of a proof. We will identify $\CC$-bimodules with endo-pseudofunctors of the presheaf 2-category $[(\CC[1])^{\op}, \Cat]$ that preserve all weighted bicolimits. For the relevant bicategorical definitions, see \cite{fibrations} and \cite{fibrations-2}.

Given a pseudofunctor $\widehat{\XX}: [(\CC[1])^{\op}, \Cat] \to [(\CC[1])^{\op}, \Cat]$ preserving weighted bicolimits, we restrict via the bicategorical Yoneda embedding to get a pseudofunctor $\XX': \CC[1] \to [(\CC[1])^{\op}, \Cat]$, which corresponds under adjunction to a $\CC$-bimodule $\XX: (\CC[1])^{\op} \times \CC[1] \to \Cat$. In the other direction, given a module $\XX: (\CC[1])^{\op} \times \CC[1] \to \Cat$, we construct a weighted bicolimit-preserving pseudofunctor $\widehat{\XX}: [(\CC[1])^{\op}, \Cat] \to [(\CC[1])^{\op}, \Cat]$ by
\[\widehat{\XX}(\YY)(\ast) = \bicolim(\YY(-), \XX(\ast, -)).\]
These two constructions give quasi-inverse equivalences of 2-categories, and take the tensor product of bimodules to the composition of pseudofunctors. Since pseudofunctors form a (strict) monoidal 2-category, we may transport this structure to the 2-category of $\CC$-bimodules to get a monoidal 2-category.
\end{proof}

Next, we prove a coherence result for bimodules.

\begin{defn}
Let $\CC$ be a strict monoidal category. A \define{strict $\CC$-bimodule} is a category $\XX$ equipped with a strict 2-functor $\CC \times \CC^{\op} \to \Aut(\XX)$.
\end{defn}

\begin{defn}
Let $\CC$ be a strict monoidal category, and let $\XX$ and $\YY$ be strict $\CC$-bimodules. A \define{strict morphism} from $\XX$ to $\YY$ is a strict 2-natural transformation $\XX \Rightarrow \YY$.
\end{defn}

\begin{thm}
\label{coherence}
Let $\CC$ be a monoidal category, and let $\XX$ be a $\CC$-bimodule. Then there is a strict $\CC$-bimodule $\st \XX$ over $\st \CC$ and an equivalence $\XX \to \st \XX$ that is equivariant with respect to the strictification maps $\CC \to \st \CC$.
\end{thm}

\begin{proof}
We consider the bicategory $\CC_{\XX}$ defined as follows.
\begin{itemize}
\item $\CC_{\XX}$ has two objects, which we will denote $\sA$ and $\sB$.
\item $\End(\sA) = \End(\sB) = \CC$.
\item $\Hom(\sA, \sB) = \XX$, while $\Hom(\sB, \sA) = \varnothing$.
\item Composition is given by the left and right actions of $\CC$ on $\XX$.
\end{itemize}
We now appeal to the coherence theorem for bicategories (see, for example, \cite[Section 2]{maclanepare}) to get a strict 2-category $\st \CC_{\XX}$ with the same object set as $\CC_{\XX}$ and a biequivalence $F: \CC_{\XX} \to \st \CC_{\XX}$. The biequivalence $F$ induces monoidal equivalences $\End(\sA) \to \st \End(\sA) = \st \CC$, $\End(\sB) \to \st \End(\sB) = \st \CC$, and $\XX \to \st \XX$ compatible with the left and right actions.
\end{proof}

\begin{defn}
Let $\GG$ be a (strict) 2-group. A $\GG$-bimodule $\XX$ is a \define{$\GG$-bitorsor} if $\XX$ is a torsor separately with respect to both the left action and the right action.
\end{defn}

\begin{prop}
Let $\XX$ and $\YY$ be $\GG$-bitorsors. Then $\XX \times_{\GG} \YY$ is a $\GG$-bitorsor.
\end{prop}

\begin{proof}
By Theorem \ref{2torsors}, $\YY$ is equivalent to $\GG$ as a left $\GG$-torsor. It follows that as a left $\GG$-module, $\XX \times_{\GG} \YY \cong \XX \times_{\GG} \GG \cong \XX \cong \GG$, so $\XX \times_{\GG} \YY$ is a left $\GG$-torsor. An identical argument shows that it is a right $\GG$-torsor.
\end{proof}

We can construct nontrivial (strict) $\GG$-bitorsors as follows. Given $\Phi \in \Aut(\GG)$, we define a bitorsor $\GG_{\Phi}$ to be $\GG$ as a left $\GG$-torsor. The right action is given by $\GG^{\op} \stackrel{\;^{-1}}{\to} \GG \stackrel{\Phi}{\to} \GG \stackrel{\cong}{\to} \Aut_{\GG}(\GG)$, where the last arrow comes from Proposition \ref{auttorsor}. Explicitly, $(g, g') \in \GG \times \GG^{\op}$ acts by $x \mapsto gx\Phi(g')$.

Our next goal is to show that every $\GG$-bitorsor is equivalent to one of the above form. To do this, we introduce the \define{adjoint autoequivalence} of $\XX$.

For a given $x \in \XX$, we define a 2-group $\Gamma_{\GG, \XX, x}$ consisting of triples $(g, g', \phi)$, where $g, g' \in \GG$ and $\phi: g \lact x \stackrel{\cong}{\to} x \ract g'$. A morphism of triples $(g_1, g_1', \phi_1) \to (g_2, g_2', \phi_2)$ is a pair of morphisms $g_1 \to g_2$, $g_1' \to g_2'$ that intertwine $\phi_1$ and $\phi_2$. The tensor product on $\Gamma_{\GG, \XX, x}$ is given by $(g_1, g_1', \phi_1) \otimes (g_2, g_2', \phi_2) = (g_1 \otimes g_2, g_1' \otimes g_2', \phi_1 \bowtie \phi_2)$, where $\phi_1 \bowtie \phi_2 = (\phi_1 g_2') \circ (g_1 \phi_2)$. Theorem \ref{coherence} guarantees that this product is coherently associative.

There are two natural projections $\pi_1, \pi_2: \Gamma_{\GG, \XX, x} \to \GG$. Each of these is an equivalence, as for any $g \in \GG$, there is a unique (up to unique isomorphism) $g' \in \GG$ with $g \lact x \cong x \ract g'$, and the freedom in choosing this isomorphism is precisely the automorphism group of $g$ itself.

\begin{defn}
The \define{adjoint autoequivalence} of the pair $(\XX, x)$ is the autoequivalence of $\GG$ given by $\Ad_{\XX, x} = \pi_1 \circ \pi_2^{-1}$.
\end{defn}

\begin{prop}
\label{bitorcoh}
Let $\XX$ be a $\GG$-bitorsor, $x \in \XX$. Then $\XX \cong \GG_{\Ad_{\XX, x}}$.
\end{prop}

\begin{proof}
We define a functor $\GG_{\Ad_{\XX, x}} \to \XX$ by $g \mapsto g \lact x$. This is a map of left $\GG$-modules by definition. The compatibility with the right action is given by the isomorphism
\begin{align*}
(g \otimes \Ad_{\XX, x}(g')) \lact x & \stackrel{\cong}{\to} g \lact (\Ad_{\XX, x}(g') \lact x)\\
& \stackrel{\cong}{\to} g \lact (x \ract g')\\
& \stackrel{\cong}{\to} (g \lact x) \ract g'.
\end{align*}
This map of bitorsors is necessarily an equivalence by Corollary \ref{mortor}.
\end{proof}

\begin{prop}
\label{bitor1}
Let $\Phi, \Psi \in \Aut(\GG)$. Then $\GG_{\Phi} \times_{\GG} \GG_{\Psi} \cong \GG_{\Psi \circ \Phi}$.
\end{prop}

\begin{proof}
We define a map $\GG \times \GG \to \GG$ by $(g, g') \mapsto g \otimes \Phi(g')$. This map extends to a $\GG$-bimodule map $\GG_\Phi \times \GG_\Psi \to \GG_{\Psi \circ \Phi}$ and descends to a $\GG$-bitorsor map $\GG_\Phi \times_{\GG} \GG_{\Psi} \to \GG_{\Psi \circ \Phi}$. This map is an equivalence by Corollary \ref{mortor}. 
\end{proof}

\begin{defn}
Let $\XX$ be a $(\GG, \HH)$-bimodule. The \define{opposite bimodule} $\XX^{-1}$ is the $(\HH, \GG)$-bimodule with the same objects as $\XX$ but where the action factors through the isomorphism $(-)^{-1}: \HH \times \GG^{\rev} \stackrel{\cong}{\to} \GG \times \HH^{\rev}$.
\end{defn}

\begin{prop}
\label{bitor2}
If $\Phi \in \Aut(\GG)$, then $(\GG_{\Phi})^{-1} \cong \GG_{\Phi^{-1}}$.
\end{prop}

\begin{proof}
The map $\Phi: \GG \to \GG$ extends to an equivalence $\GG_{\Phi^{-1}} \to (\GG_{\Phi})^{-1}$ of $\GG$-bitorsors.
\end{proof}
\begin{prop}
Let $\GG$ be a 2-group, and let $\XX$ be a $\GG$-bitorsor. Then $\XX \times_{\GG} \XX^{-1} \cong \GG \cong \XX^{-1} \times_{\GG} \XX$
\end{prop}

\begin{proof}
This result follows immediately from Corollary \ref{bitorcoh}, Proposition \ref{bitor1}, and Proposition \ref{bitor2}.
\end{proof}

\begin{prop}
\label{admor}
Let $\GG$ be a 2-group, and let $\XX$ be a $\GG$-bitorsor. Then $\Ad: \XX \to \Aut^{\otimes}(\GG)$, $x \mapsto \Ad_{\XX, x}$ is a morphism of $\GG$-bimodules, where $\GG$ acts on $\Aut^{\otimes}(\GG)$ via left and right multiplication by inner autoequivalences.
\end{prop}

\begin{proof}
We must construct natural isomorphisms $\Ad_{\XX, g \lact x}(h) \stackrel{\cong}{\to} g^{-1} \Ad_{\XX, x}(h) g$ and $\Ad_{\XX, x \ract g}(h) \stackrel{\cong}{\to} \Ad_{\XX, x}(g^{-1}hg)$.

For the former, we construct an equivalence $\Gamma_{\GG, \XX, g \lact x} \to \Gamma_{\GG, \XX, x}$ by $(h, h', \phi) \mapsto (g^{-1} h g, h', g^{-1} \phi)$. This equivalence fits into the following commutative diagram, which gives the desired isomorphism of functors.
\[\xymatrix{
& \Gamma_{\GG, \XX, g \lact x} \ar[dl]_{\cong} \ar[d]_{\cong} \ar[ddr]^{\cong}\\
\GG \ar[d]_{h \mapsto ghg^{-1}} & \Gamma_{\GG, \XX, x} \ar[dl]^{\cong} \ar[dr]_{\cong}\\
\GG & & \GG
}\]

The latter isomorphism is constructed similarly. We construct an equivalence $\Gamma_{\GG, \XX, x \ract g} \to \GG_{\GG, \XX, x}$ by $(h, h', \phi) \mapsto (h, gh'g^{-1}, \phi g^{-1})$. This equivalence fits into the following commutative diagram, which gives the desired isomorphism of functors.
\[\xymatrix{
& \Gamma_{\GG, \XX, x \ract g} \ar[ddl]_{\cong} \ar[d]_{\cong} \ar[dr]^{\cong}\\
& \Gamma_{\GG, \XX, x} \ar[dl]^{\cong} \ar[dr]_{\cong} & \GG \ar[d]^{h \mapsto g^{-1}hg}\\
\GG & & \GG
}\]
\end{proof}

\section{Bitorsors and Braided Bitorsors for Braided 2-Groups}

For general groups $G$, the adjoint automorphism associated with a $G$-bitorsor $X$ is not unique; changing the basepoint $x \in X$ will modify the automorphism by an inner automorphism. If $G$ is abelian, however, there are no nontrivial inner automorphisms, and the adjoint automorphism is unique. A similar phenomenon holds in the 2-group setting; in particular, braided 2-groups are equipped with a canonical trivialization of any inner automorphism, and so the adjoint autoequivalence is unique.

\begin{prop}
Let $\BB$ be a braided 2-group, $\XX$ a $\BB$-bitorsor, and $x \in \XX$. Then the autoequivalence $\Ad_{\XX, x}$ is independent of the choice of $x$ up to unique isomorphism.
\end{prop}

\begin{proof}
The braiding on $\BB$ provides a trivialization of the right action of $\BB$ on $\Aut^{\otimes}(\BB)$ given in Proposition \ref{admor}. Thus, the image of $x \in \XX$ under $\Ad$ is independent of $x$ up to unique isomorphism.
\end{proof}

We will abuse notation and denote by $\Ad_{\XX}$ the autoequivalence $\Ad_{\XX, x}$. We denote by ${_\BB}\TT_{\BB}$ the 3-group of $\BB$-bitorsors, and by $\Aut^{\otimes}(\BB)$ the 2-group of (not necessarily braided) autoequivalences of $\BB$.

\begin{cor}
The assignment $\XX \mapsto \Ad_{\XX}$ defines a monoidal pseudofunctor $\Ad: {_\BB}\TT_{\BB} \to \Aut^{\otimes}(\BB)$.
\end{cor}

Let $A = \pi_0(\BB)$. Then there is a monoidal pseudofunctor $\pi_0: {_\BB}\TT_{\BB} \to {_A}\TT_A$.

\begin{thm}
The maps $\Ad: {_\BB}\TT_{\BB} \to \Aut^{\otimes}(\BB)$ and $\pi_0: {_\BB}\TT_{\BB} \to {_A}\TT_A$ identify $\Aut^{\otimes}(\BB) \times_{\Aut A} {_A}\TT_A$ with the monoidal 1-truncation ${_\BB}\TT_{\BB}^{\leq 1}$ of ${_\BB}\TT_{\BB}$.
\end{thm}

\begin{proof}
It suffices to note that specifying an autoequivalence of $\BB_{\Phi}$ is determined, up to isomorphism, by a choice of $x \in \BB_{\Phi}$ (the image of the unit object $e$), or equivalently, an automorphism of the underlying $A$-bitorsor, and an automorphism of $\Phi$.
\end{proof}

\begin{defn}
A $\BB$-bitorsor $\XX$ is \define{braided} if the autoequivalence $\Ad_{\XX}$ is braided. We denote by ${_\BB}{\widetilde{\TT}}_{\BB}$ the full sub-3-group of the 3-group ${_\BB}{\TT}_{\BB}$ consisting of braided bitorsors, and by $\Aut^{\widetilde{\otimes}}(\BB)$ the 2-group of braided autoequivalences of $\BB$.
\end{defn}

\begin{cor}
\label{bratrunc}
The maps $\Ad: {_\BB}{\widetilde{\TT}}_{\BB} \to \Aut^{\widetilde{\otimes}}(\BB)$ and $\pi_0: {_\BB}{\widetilde{\TT}}_{\BB} \to {_A}\TT_A$ identify $\Aut^{\widetilde{\otimes}}(\BB) \times_{\Aut A} {_A}\TT_A$ with the monoidal 1-truncation ${_\BB}{\widetilde{\TT}}_{\BB}^{\leq 1}$ of ${_\BB}{\widetilde{\TT}}_{\BB}$.
\end{cor}

\begin{prop}
\label{bratoract}
Let $\BB$ be a braided 2-group with $\pi_1(\BB) = H$. Then $\pi_2({_\BB}{\widetilde{\TT}}_{\BB}) \cong H$, and the action of $\pi_0({_\BB}{\widetilde{\TT}}_{\BB})$ on $\pi_2$ factors through the natural map $\pi_0({_\BB}{\widetilde{\TT}}_{\BB}) \to \Aut(H)$.
\end{prop}

\begin{proof}
The first part follows from Proposition \ref{auttorsor} and the existence of the braiding. The second part follows from Proposition \ref{bitor1}.
\end{proof}

\section{Extensions and bitorsors}

Having developed the theory of braided bitorsors, we now turn to the problem of lifting group extensions to 2-group extensions. We start with a review of group extensions.

In \cite[Expos\'e VII]{sga7}, Grothendieck describes extensions of groups in terms of $G$-bitorsors. We summarize his results as follows.

\begin{thm}[Grothendieck]
\label{grothenext}
Let $G$ and $K$ be groups. Isomorphism classes of extensions of $G$ by $K$ are in bijection with isomorphism classes of monoidal functors from the discrete monoidal category $G$ to the 2-group ${_K}\TT_K$ of $K$-bitorsors. If $K = A$ is abelian, the action of $G$ on $A$ in such an extension is given by the composite map $G \to {_A}\TT_A \stackrel{\Ad}{\to} \Aut(A)$.
\end{thm}

Grothendieck's correspondence is given by assigning to an extension $1 \to K \to E \stackrel{\partial}{\to} G \to 1$ the functor $g \mapsto E_g = \partial^{-1}(g)$.

\begin{defn}
Let $G$ be a group, and let $\KK$ be a 2-group. An \define{extension} of $G$ by $\KK$ is a 2-group $\EE$ equipped with a surjective monoidal functor $\partial: \EE \to G$ and an identification of $\KK$ with $\partial^{-1}(e)$ as 2-groups.
\end{defn}

We have the following straightforward analogue of Grothendieck's theorem in this setting.

\begin{thm}
Let $G$ be a group, and let $\KK$ be a 2-group. Equivalence classes of extensions of $G$ by $\KK$ are in bijection with equivalence classes of monoidal pseudofunctors from the discrete monoidal category $G$ to the 3-group ${_\KK}\TT_{\KK}$.
\end{thm}

The analogue of an abelian group in our setting is a braided 2-group. The notion of an extension of a group by a braided 2-group was defined in \cite{dgno} as follows.

\begin{defn}
Let $G$ be a group, and let $\BB$ be a braided 2-group. Suppose we are given an action $a: G \to \Aut^{\widetilde{\otimes}}(\BB)$ of $G$ as braided autoequivalences of $\BB$. An \define{extension} of $G$ by $\BB$ with action $a$ is an underlying $2$-group extension $\EE$ equipped with an isomorphism between $a \circ \partial: \EE \to \Aut^{\otimes}(\BB)$ and the adjoint action of $\EE$ on $\BB \cong \partial^{-1}(e)$ which restricts to the trivialization of the adjoint action of $\BB$ on itself given by the braiding.
\end{defn}

We recall the following equivalent characterization of such extensions.

\begin{prop}[{\cite[Proposition E.10]{dgno}}]
Let $1 \to \BB \to \EE \to G \to 1$ be an ordinary extension of 2-groups with $\BB$ braided. This extension can be given the structure of a braided extension if and only if each $\Ad_x: \BB \to \BB$ is braided for $x \in \EE$, and such a structure is unique up to unique isomorphism.
\end{prop}

This proposition implies that the fibers of a braided extension are precisely braided $\BB$-bitorsors, so we can characterize braided extensions as follows.

\begin{cor}
Equivalence classes of braided extensions of $G$ by $\BB$ are in bijection with equivalence classes of monoidal pseudofunctors $G \to {_\BB}\widetilde{\TT}_{\BB}$.
\end{cor}

This problem of classifying braided extensions can be naturally broken up into smaller problems. First, we can limit ourselves to classifying extensions with a fixed action $a: G\to \Aut^{\widetilde{\otimes}}(\BB)$. Secondly, since every extension of 2-groups has an underlying group extension, we can study those braided extensions living above a fixed group extension. Of course, these two restrictions must be compatible: both actions of $G$ on $\BB$ and extensions of $G$ by $A = \pi_0(\BB)$ have an underlying automorphism of $A$, and these must agree. Thus, specifying these two restrictions is the same as specifying a monoidal functor from $G$ to $\Aut^{\widetilde{\otimes}}(\BB) \times_{\Aut A} {_A}\TT_A$. By Corollary \ref{bratrunc}, this is the same as giving a monoidal functor from $G$ to ${_\BB}\widetilde{\TT}_{\BB}^{\leq 1}$. Lifting the data of a group extension and an action of $G$ on $\BB$ to an extension of $G$ by $\BB$ is thus equivalent to lifting a monoidal functor $G \to {_\BB}\widetilde{\TT}_{\BB}^{\leq 1}$ to a monoidal pseudofunctor $G \to {_\BB}\widetilde{\TT}_{\BB}$. We will describe a cohomological obstruction to such a lifting in the next section.

\section{Obstructions to lifting extensions}
\label{obstruction}

Grothendieck's homotopy hypothesis, first formulated in the manuscript \cite{champs}, argues that $n$-groupoids should be ``the same as'' homotopy $n$-types. One way to make this statement a bit more precise is to say that, for any reasonable notion of $n$-groupoid, there should be a simplicial nerve functor that takes $n$-groupoids to Kan complexes whose homotopy groups vanish above level $n$ (which we will henceforth refer to simply as ``$n$-types''), and this nerve functor should be quasi-inverse to the fundamental $n$-groupoid functor.

In our current setting, we can make this precise as follows. We may ``deloop'' all of our groups, 2-groups, and 3-groups and functors between them to obtain one-object groupoids, 2-groupoids, and 3-groupoids. We will again denote by $\CC[1]$ the delooping of a group (or 2-group or 3-group) $\CC$. It is shown in \cite{berger} that there is a simplicial nerve functor for 3-groupoids. We will henceforth abuse notation and consider the deloopings $\CC[1]$ themselves as $n$-groupoids and Kan complexes interchangeably. We note that $\pi_i(\CC[1]) = \pi_{i - 1}(\CC)$.

The nerve of a braided 2-group $\BB$ is a 2-fold loop space, so we can talk not only about the delooping $\BB[1]$ but also the double delooping $\BB[2]$. Abelian groups, being infinite loop spaces, can be delooped arbitrarily many times; we will denote by $A[n]$ the $n$th delooping of an abelian group $A$ (i.e., the Eilenberg-Mac Lane space $K(A, n)$.)

As a warm-up, we will give a cohomological interpretation of Grothendieck's Theorem \ref{grothenext}. Classifying extensions of $G$ by $A$ with specified action $G[1] \to \Aut(A)[1]$  is equivalent, by Grothendieck's result, to classifying lifts as in the following diagram.
\[\xymatrix{
& {_A}\TT_A[1]\ar[d] \\
G[1]\ar@{-->}[ur] \ar[r] & \Aut(A)[1]
}\]

We note that ${_A}\TT_A[1]$ is a very special space: it is a 2-type with $\pi_1 = \Aut A$ and $\pi_2 = A$, and the action of $\Aut A$ on $A$ is the usual one. Thus, it is the space $\widehat{K}(A, 2)$ that classifies arbitrary (i.e., not necessarily simple) fibrations with fiber $K(A, 1)$. In particular, there is a universal fibration \[K(A, 1) \to P \to \widehat{K}(A, 2)\] (which was constructed in \cite{robinson} and generalizes the path-space fibration over $K(A, 2)$) from which every fibration with fiber $K(A, 1)$ is a pullback. Similarly, the spaces $\widehat{K}(A, n)$ admit universal fibrations with fiber $K(A, n - 1)$. In analogy with our notation of $A[n]$ for $K(A, n)$, we will write $\widehat{A}[n]$ for $\widehat{K}(A, n)$.

So roughly speaking, $\widehat{A}[2]$ is the classifying space for $A$-bitorsors, and we can view an extension of $G$ by $A$ as a family of $A$-bitorsors on $G[1]$, which comes via pullback from $\widehat{A}[2]$. Given an action $G[1] \to \Aut(A)[1]$, homotopy classes of lifts to $G[1] \to \widehat{A}[2]$ are precisely elements of $H^2(G, A)$ with the given action, so we recover the cohomological classification of group extensions.

We now return to the setting of braided 2-groups. We again let $\BB$ be a braided 2-group with $\pi_0(\BB) = A$ and $\pi_1(\BB) = H$. In analogy with the case of groups, we denote by $\widehat{\BB}[2]$ the (delooping of) the 3-group of \emph{braided} $\BB$-bitorsors. By Corollary \ref{bratrunc}, the following is a homotopy pullback square.

\[\xymatrix{
\widehat{\BB}[2]^{\leq 2} \ar[r] \ar[d] & \widehat{A}[2] \ar[d]\\
\Aut^{\widetilde{\otimes}}(\BB)[1] \ar[r] & \Aut (A)[1]
}\]

The data of an extension of $G$ by $A$ and a braided action of $G$ on $\BB$ with compatible underlying actions of $G$ on $A$ is thus given by a homotopy class of maps $G[1] \to \widehat{\BB}[2]^{\leq 2}$. We wish to know when this lifts to a map $G[1] \to \widehat{\BB}[2]$. Since $\widehat{\BB}[2]$ is a 3-type, the truncation map $\widehat{\BB}[2] \to \widehat{\BB}[2]^{\leq 2}$ is a fibration with fiber $\pi_3(\widehat{\BB}[2])[3] = H[3]$ by Proposition \ref{bratoract}. This fibration is a pullback from the universal fibration over $\widehat{H}[4]$. Thus, the obstruction to lifting an extension lies in $H^4(G, H)$, where the action of $G$ on $H$ is the specified one by Proposition \ref{bratoract}.

We will call the map $H^2(G, A) \to H^4(G, H)$ that assigns to a class $[\omega] \in H^2(G, A)$ the corresponding obstruction in $H^4(G, H)$ the \define{Pontryagin square}, and we will denote it by $\Pontr_{\BB, a}([\omega])$, where $a: G \to \Aut^{\widetilde{\otimes}}(\BB)$ is the action of $G$ on $\BB$. We have thus proved the following classification result.

\begin{thm}
Let $G$ be a group, and let $\BB$ be a braided 2-group. Suppose we are given an action $a: G \to \Aut^{\widetilde{\otimes}}(\BB)$ and an extension $E$ of $G$ by $\pi_0(\BB) = A$ such that the underlying actions of $G$ on $A$ coming from $a$ and $E$ agree. This data lifts to an extension of $G$ by $\BB$ if and only if $\Pontr_{\BB, a}([\omega]) = 0$, and if this obstruction vanishes, such extensions form a torsor for $H^3(G, H)$.
\end{thm}

When the action of $G$ on $\BB$ is trivial (corresponding to the case of \emph{central extensions}), we have the following diagram, where all squares are (homotopy) pullbacks.

\[\xymatrix{
H[3] \ar[d] \ar@{=}[r] & H[3] \ar[d]\\
\BB[2] \ar[d] \ar[r] & \widehat{\BB}[2] \ar[d]\\
A[2] \ar[d] \ar[r] & \widehat{\BB}[2]^{\leq 2} \ar[d] \ar[r] & \widehat{A}[2] \ar[d]\\
\ast \ar[r] &  \Aut^{\widetilde{\otimes}}(\BB)[1] \ar[r] &  \Aut(A)[1]}\]

From this diagram it follows that the fibration of $H[3]$ over $A[2]$ defining $\BB[2]$ pulls back from $H[4]$ via the map we have defined. Thus, our definition of Pontryagin square reduces to the classical Pontryagin square $A[2] \to H[4]$ corresponding to the quadratic map $A \to H$ coming from the braided 2-group $\BB$ (i.e., the Whitehead half-square map $\pi_2(\BB[2]) \to \pi_3(\BB[2])$).

\section{Cochain-level description of the Pontryagin square}

In the previous section, we defined the Pontryagin square as a ``parametrized cohomology operation'' $\Aut^{\widetilde{\otimes}}(\BB)[1] \times_{\Aut(A)[1]} \widehat{A}[2] \to \widehat{H}[4]$. In order to give an explicit description of this operation in terms of group cochains, we need first to understand the structure of $\Aut^{\widetilde{\otimes}}(\BB)$.

We recall (see \cite{joyalstreet-draft} or \cite[Section 3]{joyalstreet} for definitions and details) that the data of a braided 2-group with $\pi_0 = A$ and $\pi_1 = H$ may be presented skeletally as an abelian 3-cocycle $(A, H, h, c)$; two such abelian 3-cocycles describe the same braided 2-group if and only if their difference is an abelian 3-coboundary.

\begin{prop}
Let $\BB$ be the braided 2-group corresponding to the abelian 3-cocycle $(A, H, h, c)$. Then the 2-group $\Aut^{\widetilde{\otimes}}(\BB)$ can be described as follows.
\begin{itemize}
\item An object consists of a triple $(\phi, \psi, k)$, where $\phi \in \Aut(A)$, $\psi \in \Aut(H)$, and $k: A \times A \to H$ is a normalized 2-cochain such that
\begin{align}
\label{coch}
dk &= \psi \circ h - h \circ \phi^3\text{, and}\\
\label{coch2} \psi \circ c + k &= k \circ \tau + c \circ \phi^2,
\end{align}
where $\tau: A\times A \to A \times A$ switches the two factors.
\item A morphism $(\phi, \psi, k) \to (\phi', \psi', k')$ consists of a normalized 1-cochain $\eta: A \to H$ such that $d\eta = k - k'$. Composition of morphisms is given by addition of 1-cochains.
\item The tensor product on objects is given by \[(\phi, \psi, k) \circ (\phi', \psi', k') = (\phi \circ \phi', \psi \circ \psi', k \circ (\phi')^2 + \psi \circ k').\]
This product is strictly associative.
\end{itemize}
\end{prop}

\begin{proof}
See \cite[Proposition 14]{joyalstreet-draft}. 
\end{proof}

We can now describe actions of $G$ as braided autoequivalences of $\BB$ that lift given actions of $G$ as automorphisms of $A$ and $H$.

\begin{cor}
\label{gaction}
Fix actions of $G$ on $A$ and $H$. An action of $G$ as braided autoequivalences of $\BB$ that induces these two actions is given by the following data.
\begin{itemize}
\item For each $g \in G$, we have a normalized 2-cochain $k_g: A \times A \to H$ satisfying \eqref{coch} and \eqref{coch2}.
\item For each pair $g_1, g_2 \in G$, we have a normalized 1-cochain $\theta_{g_1, g_2}: A \to H$ such that 
\begin{align*}
d\theta_{g_1, g_2} &= k_{g_1g_2} - k_{g_1} \circ (g_2 \lact -)^2 - g_1 \lact k_{g_2}\text{, and}\\
\theta_{g_1g_2, g_3} + \theta_{g_1, g_2} \circ (g_3 \lact -) &= \theta_{g_1, g_2g_3} + g_1 \lact \theta_{g_2, g_3}.
\end{align*}
\end{itemize}

An isomorphism $(k, \theta) \stackrel{\cong}{\Rightarrow} (k', \theta')$ of such actions is given by the data of, for each $g \in G$, a normalized 1-cochain $\eta_g: A \to H$ such that $d\eta_g = k_g' - k_g$, and such that for each pair $g_1, g_2 \in G$,
\[\theta'_{g_1, g_2} - \theta_{g_1, g_2} = \eta_{g_1g_2} - \eta_{g_1} \circ (g_2 \lact -) - g_1 \lact \eta_{g_2}.\]

\end{cor}

Since $\Aut^{\widetilde{\otimes}}(\BB)[1] \times_{\Aut(A)[1]} \widehat{A}[2]$ is the 2-truncation of $\widehat{\BB}[2]$, maps from $G[1]$ to this 2-truncation can be presented as quintuples $(\phi, \psi, k, \theta, \omega)$, where $\phi$ and $\psi$ are the actions of $G$ on $A$ and $H$, $k$ and $\theta$ are the data described in Corollary \ref{gaction}, and $\omega: G \times G \to A$ is a 2-cocycle for $G$ with coefficients in $A$.

We wish to write down an obstruction cocycle in $Z^4(G, H)$ representing the Pontryagin square in terms of the data $(A, H, h, c)$ of the braided 2-group $\BB$ and $(\phi, \psi, k, \theta, \omega)$ of a map $G[1] \to \widehat{\BB}[2]^{\leq 2}$.

\begin{prop}
Let $\BB$ be a braided 2-group with data $(A, H, h, c)$, $a: G \to \Aut^{\widetilde{\otimes}}(\BB)$ an action with data $(\phi, \psi, k, \theta)$, and $[\omega] \in H^2_{\phi}(G, A)$. The Pontryagin square $\Pontr_{\BB, a}([\omega])$ is represented by the 4-cocycle
\begin{align*}
\pi(g_1, g_2, g_3, g_4) &= c(\omega(g_1, g_2), g_1g_2 \lact \omega(g_3, g_4))\\
&\hphantom{=} + h(g_1g_2 \lact \omega(g_3, g_4), \omega(g_1, g_2), \omega(g_1g_2, g_3g_4))\\
&\hphantom{=} - h(g_1g_2 \lact \omega(g_3, g_4), g_1 \lact \omega(g_2, g_3g_4), \omega(g_1, g_2g_3g_4))\\
&\hphantom{=} + h(g_1 \lact \omega(g_2, g_3), g_1 \lact \omega(g_2g_3, g_4), \omega(g_1, g_2g_3g_4))\\
&\hphantom{=} - h(g_1 \lact \omega(g_2, g_3), \omega(g_1, g_2g_3), \omega(g_1g_2g_3, g_4))\\
&\hphantom{=} + h(\omega(g_1, g_2), \omega(g_1g_2, g_3), \omega(g_1g_2g_3, g_4))\\
&\hphantom{=} - h(\omega(g_1, g_2), g_1g_2 \lact \omega(g_3, g_4), \omega(g_1g_2, g_3g_4))\\
&\hphantom{=} + \theta_{g_1, g_2}(\omega(g_3, g_4))\\
&\hphantom{=} - k_{g_1}(g_2 \lact \omega(g_3, g_4), \omega(g_2, g_3g_4))\\
&\hphantom{=} + k_{g_1}(\omega(g_2, g_3), \omega(g_2g_3, g_4)).
\end{align*}
\end{prop}

\begin{proof}
The cocycle $\omega: G \times G \to A$ describes an ``extension with section'': an extension $0 \to A \to E \stackrel{\partial}{\to} G \to 1$ equipped with a section $s: G \to E$. The multiplication on $E$ is described by $\omega$: we have \[(a_1 + s(g_1)) \cdot (a_2 + s(g_2)) = a_1 + g_1 \lact a_2 + \omega(g_1, g_2) + s(g_1g_2).\]
The cocycle condition ensures that this multiplication is associative.

The cocycle $\omega$ lifts uniquely to $\Omega: G \times G \to \BB$. We would like to understand when this cocycle describes an ``extension with section'' of $G$ by $\BB$, but this requires additional data: a collection of isomorphisms
\[\Upsilon_{g_1, g_2, g_3}: \Omega(g_1, g_2) \otimes \Omega(g_1 g_2, g_3) \Rightarrow g_1 \lact \Omega(g_2, g_3)) \otimes \Omega(g_1, g_2 g_3),\]
which corresponds to the associator. Such isomorphisms always exist because $\omega$ is a cocycle.

In order to be a true extension, this associator must satisfy the pentagon identity, which translates to the commutativity of the following diagram.

\tiny
\begin{equation*}
\label{cocycle}
\xymatrix{
\Omega(g_1, g_2) \otimes (g_1g_2 \lact \Omega(g_3, g_4)) \otimes \Omega(g_1g_2, g_3g_4) \ar[r]^-{c \otimes \id} & (g_1g_2 \lact \Omega(g_3, g_4)) \otimes \Omega(g_1, g_2) \otimes \Omega(g_1g_2, g_3g_4) \ar[d]^{\id \otimes \Upsilon}\\
\Omega(g_1, g_2) \otimes \Omega(g_1g_2, g_3) \otimes \Omega(g_1g_2g_3, g_4) \ar[u]^{\id \otimes \Upsilon} \ar[d]_{\Upsilon \otimes \id} & g_1 \lact (g_2 \lact \Omega(g_3, g_4) \otimes  \Omega(g_2, g_3g_4)) \otimes \Omega(g_1, g_2g_3g_4)\\
(g_1 \lact \Omega(g_2, g_3)) \otimes \Omega(g_1, g_2g_3) \otimes \Omega(g_1g_2g_3, g_4) \ar[r]_-{\id \otimes \Upsilon} & g_1 \lact (\Omega(g_2, g_3) \otimes \Omega(g_2g_3, g_4)) \otimes \Omega(g_1, g_2g_3g_4) \ar[u]_{g_1 \lact \Upsilon \otimes \id}
}
\end{equation*}
\normalsize

Translating this diagram into the skeletal setting yields the desired formula for $\pi$ (plus a coboundary corresponding to $\Upsilon$).

\end{proof}

\begin{remark}
It is not \textit{a priori} clear, without our earlier homotopy theoretic discussion, that the cochain $\pi$ is even a cocycle, much less that is well-defined (up to changing it by a coboundary) independently of the choice of rigidifying skeletal data. Indeed, direct computations of these facts at the level of cochains are extremely lengthy.
\end{remark}

\bibliographystyle{amsalpha}
\nocite{*}
\bibliography{extensions}

\providecommand{\bysame}{\leavevmode\hbox to3em{\hrulefill}\thinspace}
\providecommand{\MR}{\relax\ifhmode\unskip\space\fi MR }
\providecommand{\MRhref}[2]{%
  \href{http://www.ams.org/mathscinet-getitem?mr=#1}{#2}
}
\providecommand{\href}[2]{#2}
\begin{thebibliography}{DGNO10}

\bibitem[Bau91]{baues}
Hans-Joachim Baues, \emph{Combinatorial homotopy and $4$-dimensional
  complexes}, de Gruyter Expositions in Mathematics, vol.~2, Walter de Gruyter
  \& Co., Berlin, 1991, With a preface by Ronald Brown.

\bibitem[Ber99]{berger}
Clemens Berger, \emph{Double loop spaces, braided monoidal categories and
  algebraic $3$-type of space}, Higher homotopy structures in topology and
  mathematical physics {(Poughkeepsie,} {NY,} 1996), Contemp. Math., vol. 227,
  Amer. Math. Soc., Providence, {RI}, 1999, pp.~49--66.

\bibitem[BL04]{baezlauda}
John~C. Baez and Aaron~D. Lauda, \emph{Higher-dimensional algebra {V}:
  2-groups}, Theory and Applications of Categories \textbf{12} (2004),
  423--491.

\bibitem[CMV02]{lecreurer}
I.~J.~Le Creurer, F.~Marmolejo, and E.~M. Vitale, \emph{Beck's theorem for
  pseudo-monads}, Journal of Pure and Applied Algebra \textbf{173} (2002),
  no.~3, 293--313.

\bibitem[DGNO10]{dgno}
Vladimir Drinfeld, Shlomo Gelaki, Dmitri Nikshych, and Victor Ostrik, \emph{On
  braided fusion categories. {I}.}, Selecta Mathematica. New Series \textbf{16}
  (2010), no.~1, 1--119.

\bibitem[ENO10]{eno}
Pavel Etingof, Dmitri Nikshych, and Victor Ostrik, \emph{Fusion categories and
  homotopy theory}, Quantum Topology \textbf{1} (2010), no.~3, 209--273, With
  an appendix by Ehud Meir.

\bibitem[Gro83]{champs}
A.~Grothendieck, \emph{{Pursuing stacks}}, Unpublished letter to D. Quillen,
  1983.

\bibitem[JS86]{joyalstreet-draft}
Andr\'e Joyal and Ross Street, \emph{Braided monoidal categories}, Tech. Report
  860081, Macquarie Mathematics Reports, November 1986.

\bibitem[JS93]{joyalstreet}
Andr\'e Joyal and Ross Street, \emph{Braided tensor categories}, Advances in
  Mathematics \textbf{102} (1993), 20--78.

\bibitem[MLP85]{maclanepare}
Saunders Mac~Lane and Robert Par\'{e}, \emph{Coherence for bicategories and
  indexed categories}, Journal of Pure and Applied Algebra \textbf{37} (1985),
  no.~1, 59{\textendash}80.

\bibitem[Rob72]{robinson}
C.~A. Robinson, \emph{{Moore-Postnikov} systems for non-simple fibrations},
  Illinois Journal of Mathematics \textbf{16} (1972), 234--242.

\bibitem[SGA72]{sga7}
\emph{Groupes de monodromie en g\'eom\'etrie alg\'ebrique. i}, 1972,
  S\'eminaire de G\'eom\'etrie Alg\'ebrique du {Bois-Marie} 1967--1969 {(SGA} 7
  I), Dirig\'e par A. Grothendieck. Avec la collaboration de M. Raynaud et D.
  S. Rim.

\bibitem[Str80]{fibrations}
Ross Street, \emph{Fibrations in bicategories}, Cahiers de Topologie et
  G\'eom\'etrie Diff\'erentielle \textbf{21} (1980), no.~2, 111--160.

\bibitem[Str87]{fibrations-2}
\bysame, \emph{Correction to: {``Fibrations} in bicategories'' {[Cahiers}
  {T}opologie {G}\'eom.\ {D}iff\'erentielle \ \textbf{21} (1980), no.\ 2,
  111--160; {MR0574662} (81f:18028)]}, Cahiers de Topologie et G\'eom\'etrie
  Diff\'erentielle Cat\'egoriques \textbf{28} (1987), no.~1, 53--56.

\bibitem[Wei95]{weibel}
Charles Weibel, \emph{An introduction to homological algebra}, Cambridge
  Studies in Advanced Mathematics, no.~38, Cambridge University Press, 1995.

\end{thebibliography}

\end{document}